\documentclass{article}

\usepackage{amssymb,amsthm,amsmath,mathrsfs}
\usepackage{amsfonts}
\usepackage{graphicx}
\usepackage{abstract}
\usepackage[english]{babel}
\usepackage{cite}

\newtheorem{theorem}{Theorem}

\newtheorem{corollary}[theorem]{Corollary}

\newtheorem{lemma}[theorem]{Lemma}
\newtheorem{proposition}[theorem]{Proposition}
\theoremstyle{definition}
\newtheorem{example}[theorem]{Example}
\newtheorem{definition}[theorem]{Definition}
\newtheorem{remark}[theorem]{Remark}

\newcommand{\rank}{\operatorname{rk}}
\newcommand{\im}{\operatorname{im}}

\newcommand{\Sh}{\operatorname{Sh}}
\newcommand{\keywords}{\noindent \textit{Keywords:} }
\newcommand{\MSC}{\noindent \textit{2010 MSC:} }

\begin{document}

\title{Regular blocks and Conley index of isolated invariant continua in surfaces
\footnote{The author is  supported by the FPI grant BES-2013-062675 and by MINECO (MTM2012-30719).} }

\author{H\'ector Barge\footnote{Facultad de C. C. Matem\'aticas, Universidad Complutense de Madrid, Madrid 28040, Spain; email:hbarge@ucm.es}}

\date{}

\maketitle


\begin{abstract}
In this paper we study topological and dynamical features of isolated invariant continua of continuous flows defined on surfaces.  We show that near an isolated invariant continuum the flow is topologically equivalent to a $C^1$ flow. We deduce that  isolated invariant continua in surfaces have the shape of finite polyhedra. We also show the existence of  \emph{regular isolating blocks} of isolated invariant continua and we use them to compute their Conley index provided that we have some knowledge about the truncated unstable manifold. We also see that the ring structure cohomology index of an isolated invariant continuum in a surface determines its Conley index. In addition, we study the dynamics of non-saddle sets, preservation of topological and dynamical properties by continuation and we give a topological classification of isolated invariant continua which do not contain fixed points and, as a consequence, we also classify isolated minimal sets.  
\end{abstract}

\MSC{34C45, 37G35, 58J20}
\vspace{3pt}


\section{Introduction}

In this paper we study topological and dynamical features of isolated invariant continua of continuous flows $\varphi :M\times \mathbb{R}%
\rightarrow M$ defined on surfaces. By a surface $M$ we mean a connected 2-manifold without boundary. To avoid trivial cases, when we refer to an isolated invariant continuum $K$, it will be implicit that it is a proper subset of $M$, i.e. $\emptyset\neq K\subsetneq M$. 

The paper is structured as follows. In Section~\ref{sec:1} we show that near an isolated invariant continuum $K$ the flow is topologically equivalent to a $C^1$ flow and, as a consequence, $K$ admits a basis of neighborhoods comprised of what we call \emph{isolating block manifolds}. The main result of this section is Theorem~\ref{polyhedron} which establishes that an isolated invariant continuum $K$ of a flow on a surface must have the shape of a finite polyhedron.  Besides, we characterize the \emph{initial sections} of the truncated unstable manifold $W^{u}(K)-K$ introduced in \cite{bargeunstable2014}.  Section~\ref{sec:2} is devoted to prove the main results of the paper which are Theorem~\ref{blockbasis} where the existence of the so-called \emph{regular isolating blocks} of isolated invariant continua on surfaces is established and Theorem~\ref{main} which establishes a complete classification of the possible values taken by the Conley index of $K$. In particular, it is proven that the Conley of $K$ is the pointed homotopy type of a wedge of circumferences if $K$ is neither an attractor nor a repeller, the pointed homotopy type of a disjoint union of a wedge of circumferences and an external point (which is the base point) if $K$ is an attractor and the pointed homotopy of a wedge of circumferences and a closed surface if $K$ is a repeller. Both the number of circumferences in the wedge and the corresponding genus of the surface in the case of repellers are determined by the first Betti number of $K$  and the knowledge of an initial section of $W^u(K)-K$. The existence of regular isolating blocks plays a key role in our proof of this classification. In Section~\ref{sec:3} we prove Theorem~\ref{ring} which is a classification of the Conley index of $K$ in terms of  the ring structure of the cohomology index. Finally, Section~\ref{sec:4} is devoted to some applications of the previous results. The main results of this section are Theorem~\ref{cont2} and Theorem~\ref{fixed}. Theorem~\ref{cont2} studies the preservation of some topological and dynamical properties by continuation. For instance, it is proven that if $(K_\lambda)_{\lambda\in I}$ is a continuation of an attractor (resp. repeller) $K_0$ then, for each $\lambda$, $K_\lambda$ must have a component $K^1_\lambda$ which is an attractor (resp. repeller) with the same shape of $K_0$. It is also proven that the property of being saddle is preserved by continuation for small values of the parameter and that if $K_\lambda$ is a continuum for each $\lambda$, the property of being non-saddle is preserved if and only if the shape is preserved. On the other hand, Theorem~\ref{fixed} establishes that if an isolated invariant continuum in a surface does not have fixed points it must be non-saddle and either a limit cycle, a closed annulus bounded by two limit cycles or a M\"obius strip bounded by a limit cycle.  A nice consequence of this result is Corollary~\ref{minimal} which establishes that a minimal isolated invariant continuum in a surface must be either a fixed point or a limit cycle.    

We shall use through the paper the standard notation and terminology in the
theory of dynamical systems. By the \textit{omega-limit }of a set $X\subset M$ \ we understand the set $\omega (X)=\bigcap_{t>0}
\overline{X\cdot \lbrack t,\infty)}$ while the \textit{negative omega-limit} is the
set $\omega ^{\ast }(X)=\bigcap_{t>0}\overline{X\cdot (-\infty ,-t]}$. The \emph{unstable manifold} of an invariant compactum $K$ is defined as the set $W^u(K)=\{x\in M \mid \emptyset\neq\omega^*(x)\subset K\}$. Similarly the \emph{stable manifold} $W^s(K)=\{x\in M\mid \emptyset\neq\omega(x)\subset K\}$.
For us, an \textit{attractor} is an \textit{asymptotically stable}
compactum and a \textit{repeller} is an asymptotically stable compactum for the reverse flow. 

We shall assume in the paper some knowledge of the Conley index theory of
isolated invariant compacta of flows. These are compact invariant sets $K$
which possess a so-called isolating neighborhood, that is, a compact
neighborhood $N$ such that $K$ is the maximal invariant set in $N$, or
setting

\[
N^{+}=\{x\in N:x[0,+\infty )\subset N\};\text{ \ \ \ }N^{-}=\{x\in
N:x(-\infty ,0]\subset N\}; 
\]

such that $K=N^{+}\cap N^{-}$. We shall make use of a special type of
{iso\-la\-ting} neighborhoods, the so-called isolating blocks, which have good
topological properties. More precisely, an isolating block $N$ is an
isolating neighborhood such that there are compact sets $N^{i},N^{o}\subset
\partial N$, called the entrance and exit sets, satisfying
\begin{enumerate}
\item $\partial N=N^{i}\cup N^{o}$,

\item for every $x\in N^{i}$ there exists $\varepsilon >0$ such that $%
x[-\varepsilon ,0)\subset M-N$

and for every $x\in N^{o}$ there exists $\delta >0$ such that $%
x(0,\delta ]\subset M-N$,

\item for every $x\in \partial N-N^{i}$ there exists $\varepsilon >0$ such that $%
x[-\varepsilon ,0)\subset \mathring{N}$

and for every $x\in \partial N-N^{o}$ there exists $\delta >0$ such
that $x(0,\delta ]\subset \mathring{N}$.
\end{enumerate}

These blocks form a neighborhood basis of $K$ in $M$. Associated to an isolating block $N$ there are defined two continuous functions
\[
t^s:N-N^+\to[0,+\infty),\quad t^i:N-N^-\to(-\infty,0]
\]
given by
\[
t^s(x):=\sup\{t\geq 0\mid x[0,t]\subset N\},\quad t^i(x):=\inf\{t\leq 0\mid x[t,0]\subset N\}.
\]
These functions are known as the \emph{exit time} and the \emph{entrance time} respectively.
We shall also use the
notation $n^{+}=N^{+}\cap \partial N$ and $n^{-}=N^{-}\cap \partial N$. If the flow is differentiable, the isolating blocks can be chosen to be
manifolds which contain $N^{i}$ and $N^{o}$ as submanifolds
of their boundaries and such that $\partial N^{i}=\partial N^{o}=N^{i}\cap
N^{o}$. This kind of isolating blocks will be called \emph{isolating block manifolds}. For flows defined on surfaces, the exit set $N^{o}$ of an isolating block manifold
is the disjoint union of a finite number of intervals $J_{1},\ldots,J_{m}$ and
circumferences $C_{1},\ldots,C_{n}$ and the same is true for the entrance set $%
N^{i}$. 

 We also recall some dynamical concepts introduced in \cite{bargeunstable2014}. If $K$ is an isolated invariant set, the subset $W^u(K)-K$ will be referred as the \emph{truncated unstable manifold} of $K$. A compact section $S$ of $W^u(K)-K$, i.e., a compact subset $S$ of $W^u(K)-K$ such that for each $x\in W^u(K)-K$ there exists a unique $t\in\mathbb{R}$ such that $xt\in S$, is said to be \emph{initial} provided that $\omega^*(S)\subset K$. Notice that, if $S$ and $S'$ are two different initial sections they are homeomorphic. The subset $I_S^u(K)=S(-\infty,0]$ is the so-called $S$-\emph{initial part of the truncated unstable manifold} $W^u(K)-K$ and it turns out to be homeomorphic to the product $S\times (-\infty,0]$. For instance, if $N$ is an isolating block, $n^-$ is an initial section and $N^-$ agrees with $I^u_{n^-}(K)$.  We would like to point out that all this concepts may be dualized for the stable manifold in the obvious way.  
 
We shall also make use of a classical result of C. Guti\'{e}rrez about smoothing of 2-dimensional flows.

\begin{theorem}[Guti\'{e}rrez \cite{Gutierrez}] Let $\varphi :M\times \mathbb{R}\rightarrow M$ be a continuous flow on a compact $C^{\infty }$ two-manifold $M$. Then there
exists a $C^{1}$ flow $\psi $ on $M$ which is topologically equivalent to $\varphi $. Furthermore, the following conditions are equivalent:
\begin{enumerate}
\item any minimal set of $\varphi $ is trivial;

\item $\varphi $ is topologically equivalent to a $C^{2}$ flow;

\item $\varphi $ is topologically equivalent to a $C^{\infty }$ flow.
\end{enumerate}
\end{theorem}

By a trivial minimal set we understand a fixed point, a closed trajectory or the whole manifold if $M$ is the $2$-dimensional torus and $\varphi$ is topologically equivalent to an irrational flow.

We adopt in the paper a topological viewpoint close to the one adopted, for example, in the papers \cite{JaimeNLA,GabitesTrans,SanjurjoJDE}. Homotopy and homology theory play an important role dealing with the Conley index theory. In particular, we shall  use through the paper \v Cech cohomology $\check{H}^*$ and singular homology $H_*$ and cohomology $H^*$, all of them with $\mathbb{Z}_2$ coefficients. We recall that \v Cech and singular homology theories agree when working with spaces with good local behaviour such as manifolds, polyhedra and CW-complexes. We define the $i$-dimensional Betti number of a topological space $X$, $\beta_i(X)$ as the dimension of the vector space $\check{H}^i(K)$. The Euler characteristic $\chi(X)$, when defined, is the alternated sum of the Betti numbers. The 
\textit{Conley index} $h(K)$ of an isolated invariant set $K$ is defined as
the homotopy type of the pair $(N/N^{o},[N^{o}]),$ where $N$ is any
isolating block of $K$. A crucial fact concerning the definition is, of
course, that this homotopy type does not depend on the particular choice of $%
N$. We will also make use of the \emph{cohomology index} $CH^*(K)$ defined as $H^*(N/N^o,[N^o])\cong H^*(N,N^o)$. We refer the reader to \cite{Conley_mono, Conley_Isolated_1971, Mischaikow, Salamon} for information about the Conley index theory and to \cite{SanjurjoJDE, WojcikNLA} to see recent applications of the Conley index techniques to some problems in ecology.
 There is a form of homotopy which has proved to be the most convenient for the study of the global
{to\-po\-lo\-gi\-cal} properties of the invariant spaces involved in dynamics, namely
the \textit{shape theory} introduced and studied by Karol Borsuk. We do not
use shape theory in this paper. However, it is convenient to know that some
topological properties of continua in surfaces have a very nice interpretation in
terms of shape. Two compacta are said to be of the same shape if they have
the same homotopy type in the homotopy theory of Borsuk (or shape theory).
The following results from \cite{Gabites_Aplicaciones_2009, gabites} will be useful in the sequel. 

\begin{theorem}\label{jaime1}
Let $K$ be a compactum contained in the interior of a compact $2$-manifold $M$. If the inclusion $i:K\hookrightarrow M$ induces isomorphisms $i^*:\check{H}^k(M)\to \check{H}^k(K)$ for $0\leq k\leq 2$, then it is a shape equivalence.
\end{theorem}

\begin{corollary}\label{jaime2}
Let $K$ be a continuum contained in the interior of a $2$-manifold $M$. If $\check{H}^2(K)=0$ and $\check{H}^1(K)$ is finitely generated, then $K$ has the shape of a wedge of $\beta_1(K)$ circumferences.
\end{corollary}

Notice that if $M$ is a compact and connected 2-manifold with boundary and $K$ is a subcontinuum contained in its interior, it would be enough $i^*:\check{H}^1(M)\to \check{H}^1(K)$ to be an isomorphism to meet the assumptions of Theorem~\ref{jaime1} and, hence, to ensure that the inclusion is a shape equivalence. On the other hand, if we only consider proper subcontinua contained in the interior of connected 2-manifolds, Corollary~\ref{jaime2} ensures that $\beta_1(K)$, when finitely generated, determines the shape of $K$. These facts can be easily seen using Alexander duality. 

We are also going to make use of the fact, proved in \cite{Gabites_Aplicaciones_2009}, that if $K$ is a continuum contained in the interior of a 2-manifold $M$ and  $N_1$ and $N_2$ are connected submanifolds of $M$ which are neighborhoods of $K$ in $M$ such that the inclusions $i_k:K\hookrightarrow N_k$ are shape equivalences, then $N_1$ and $N_ 2$ are homeomorphic.  

Although we do not make use of shape theory in our proofs, we may occasionally refer to these
theorems and to the terminology derived from it to make it clear that some of
the results can be interpreted in that context. For a complete treatment of
shape theory we refer the reader to \cite{Borsuk, 
Dydak-Segal, Mardesic-Segal, SanjurjoTransactions}. The
use of shape in dynamics is illustrated by the papers \cite{GMRS, SanjurjoLN, Hastings,  Kapitanski, Robbin-Salamon,RobinsonShape,JaimeRACSAM}. For information about
basic aspects of {dy\-na\-mi\-cal} systems we recommend \cite{Bhatia-SZ, Robinsondynamics, Temam}. We also recommend the books written by Hatcher
\cite{HatcherAT} and Spanier \cite{Spanier} for questions regarding algebraic topology and the book \cite{Massey} and the paper \cite{richardson1963} as references about the topology of surfaces.

\section{Isolating blocks in surfaces}\label{sec:1}
In this section we study the structure of a flow defined on a surface near an isolated invariant continuum $K$. In particular we will see that $K$ admits a neighborhood in which the flow topologically equivalent to a $C^1$ flow. From this fact we will deduce that $K$ has the shape of a finite polyhedron. 

The next result states some useful properties of isolating blocks which will be exploited through the paper.

\begin{lemma}\label{general}
Suppose that $K$ is an isolated invariant continuum of a flow on a manifold and that $N$ is a connected isolating block manifold of $K$. Then
\begin{enumerate}
\item[a)] Each component of $N^o$ must contain some component of $n^-$,
\item[b)] $n^-$ has a finite number of components, and
\item[c)] if $x_0$ is a point in $N^o-n^-$ and $U$ a compact neighborhood of $x_0$ in $N^o-n^-$ then, the set 
\[
W=\bigcup_{x\in U}x[t^i(x),0]
\] 
is homeomorphic to the product $U\times[0,1]$ via a homeomorphism which carries each  trajectory segment $x[t^i(x),0]$ to the fiber $\{x\}\times[0,1]$.
\end{enumerate}
\end{lemma}

\begin{proof}
Since the inclusion $K\hookrightarrow N^-$ is a shape equivalence \cite{Kapitanski}, a straightforward application of the five lemma gives that $\check{H}^k(N,K)\cong\check{H}^k(N,N^-)$. In addition, the inclusion $N^-\cup N^o\hookrightarrow N$ is also a shape equivalence (see \cite{sanjurjomorse2003}) and, reasoning as before, it follows that $\check{H}^k(N,N^-)\cong\check{H}^k(N^-\cup N^o,N^-)$. On the other hand, by the strong excision property of \v Cech cohomology 
\begin{eqnarray*}
\check{H}^k(N^-\cup N^o,N^-)&\cong&\check{H}^k\left(\frac{N^-\cup N^o}{N^-},[N^-]\right)\\
&\cong & \check{H}^k\left(\frac{N^o}{n^-},[n^-]\right)\\
&\cong & \check{H}^k(N^o,n^-).
\end{eqnarray*}

Since $N$ and $K$ are connected, $\{0\}=\check{H}^0(N,K)=\check{H}^0(N^o,n^-)$ and, hence, from the long exact sequence of cohomology of the pair $(N^o,n^-)$ we get that the homomorphism
\[
\check{H}^0(N^o)\to\check{H}^0(n^-)
\]
induced by the inclusion $n^-\hookrightarrow N^o$ is a monomorphism. This proves a).
 
Consider the long exact sequence of reduced \v Cech cohomology of the pair $(N,K)$
\[
0\to \check{H}^1(N,K)\to\check{H}^1(N)\to\check{H}^1(K)\to\check{H}^2(N,K)\to \cdots
\]
Since $N$ is a manifold, then $\check{H}^1(N)$ agrees with $H^1(N)$ and, hence, it is finitely generated. Thus, from the exact sequence we get that $\check{H}^1(N,K)$ is also finitely generated. As a consequence, $\check{H}^1(N^o,n^-)$ is finitely generated being isomorphic to $\check{H}^1(N,K)$. Moreover, since $\check{H}^0(N^o,n^-)=\{0\}$, the long exact sequence of the pair  $(N^o,n^-)$ splits into the short exact sequence
\[
0\to \check{H}^0(N^o)\to\check{H}^0(n^-)\to\im\delta\to 0
\]
where $\delta: \check{H}^0(n^-)\to \check{H}^1(N^o,n^-)$ is the coboundary homomorphism. In addition, the groups $\check{H}^0(N^o)$ and $\im\delta$ are finitely generated since $N^o$ has a finite number of components being a compact manifold and $\im\delta$ being a subgroup of the finitely generated group $\check{H}^1(N^o,n^-)$. Therefore, $\check{H}^0(n^-)$ is finitely generated. This proves b).

Let $x_0\in N^o-n^-$ and $U$ be a compact neighborhood of $x_0$ in $N^o-n^-$. Consider for each $x\in U$ the linear homeomorphism $\sigma_x:[0,1]\to[t^i(x),0]$ given by $\sigma_x(s)=t^i(x)(1-s)$. We define $h:U\times[0,1]\to W$ as $h(x,s)=x\sigma_x(s)$ which is clearly a bijection. See that $h$ is continuous. Let $(x_n)$ and $(s_n)$ sequences in $U$ and $[0,1]$ convergent to $\bar{x}\in U$ and $\bar{s}\in [0,1]$ respectively. Then, $\sigma_{x_n}(s_n)=t^i(x_n)(1-s_n)$, which by the continuity of $t^i$ converges to $\sigma_{\bar{x}}(\bar{s})$ and, hence, $h(x_n,s_n)$ converges to $h(\bar{x},\bar{s})$ by the continuity of the flow. Therefore, $h$ is continuous. Let us see that $h^{-1}$ is also continuous. Consider a sequence $(y_n)$ of points in $W$ convergent to a certain $\bar{y}\in W$. Each $y_n$ is of the form $x_n\sigma_{x_n}(s_n)$ and $\bar{y}=\bar{x}\sigma_{\bar{x}}(\bar{s})$ respectively, where,  $x_n$, $\bar{x}\in U$ and $s_n$, $\bar{s}\in [0,1]$. See that $x_n$ converge to $\bar{x}$ and $s_n$ converge to $\bar{s}$. Since $U$ and $[0,1]$ are compact, we can choose subsequences $x_{n_k}\to x'$ and $s_{n_k}\to s'$. Besides, the continuity of $h$ guarantees $h(x_{n_k},s_{n_k})\to h(x',s')$. But, on the other hand, $h(x_{n_k},s_{n_k})=x_{n_k}\sigma_{x_{n_k}}(s_{n_k})\to\bar{y}$. As a consequence we get that $\bar{y}=h(x',s')$, leading to $\bar{x}\sigma_{\bar{x}}(\bar{s})=x'\sigma_{x'}(s')$. Then, it follows that $\bar{x}=x'$ and $\bar{s}=s'$. Indeed, suppose, arguing by contradiction, that $\bar{x}\neq x'$, then, assuming that the absolute value of $\sigma_{\bar{x}}(\bar{s})$ is greater than or equal to $\sigma_{x'}(s')$ we would have that $\bar{x}(\sigma_{\bar{x}}(\bar{s})-\sigma_{x'}(s'))=x'$ and, since $(\sigma_{\bar{x}}(\bar{s})-\sigma_{x'}(s'))\in(t^i(\bar{x}),0]$, it follows that either $\bar{x}=x'$ or $x'$ is point of internal tangency in contradiction with the definition of isolating block. It also follows that $\bar{s}=s'$ since, if not, the trajectory of $\bar{x}$ would be periodic and, thus, $\bar{x}$ would be a point of internal tangency. We have proved that every convergent subsequence of $(x_n)$ converge to $\bar{x}$ and every convergent subsequence of $(s_n)$ converge to $\bar{s}$. As a consequence, since $U$ and $S$ are compact, $x_n\to\bar{x}$ and $s_n\to\bar{s}$ . This proves c).
\end{proof}

From now on we will focus on flows defined on surfaces. The next result is a local version of classical Guti\'errez' Theorem. The proof is a kind of mixture of some ideas from \cite{Morón_Topology_2007} and \cite{JaimeRACSAM}.

\begin{lemma}\label{smooth}
Let $\varphi:M\times\mathbb{R}\to M$ be a flow defined on a surface and $K$ be an isolated invariant continuum. Then, $\varphi$ is topologically equivalent to a $C^1$ flow near $K$. Moreover, $K$ admits a basis of neighborhoods comprised of isolating block manifolds.
\end{lemma}
\begin{proof}
We will start the proof by showing that $K$ admits a neighborhood basis comprised of surfaces with boundary. Indeed, since $M$ is a surface, we may assume without loss of generality that $M$ is $C^\infty$ (see \cite{hatcher}). Consider the continuous map $d_K(x)=d(x,K)$. Now, fixed $\varepsilon>0$ we can find a $C^\infty$ function $\delta_K:M\to [0,+\infty)$ such that $d_K\leq \delta_K\leq d_K+\varepsilon/3$ (see \cite[Exercise~36, p. 152]{Outerelo}). We choose $\varepsilon$ in such a way that $\varepsilon\in d_K(M)$. As a consequence, $\delta_K(M)\supset [\varepsilon/3,2\varepsilon/3)$ and by Sard's Theorem \cite{Milnor} there exists a regular value $a\in(\varepsilon/3,2\varepsilon/3)$. Then, $\delta_K^{-1}((-\infty,a])$ is a compact 2-manifold with boundary \cite{Milnor}. It is clear that $K$ is contained in the interior of $\delta_K^{-1}((-\infty,a])$ since, if $x\in K$, $\delta_K(x)\leq \varepsilon/3<a$. Therefore, choosing $N$ as the component of $\delta_K^{-1}((-\infty,a])$ containing $K$ we have found the desired neighborhood. Since the choice of $\varepsilon$ was arbitrary, the claim follows.

On the other hand, since we can find a surface neighborhood $N$ of $K$ as close to $K$ as desired, we can choose it to be an isolating neighborhood. Let $\widehat{N}$ be the closed surface obtained by capping each boundary component of $N$ with a disk. By the Keesling reformulation of Beck's Theorem \cite{Keesling} we can obtain a flow $\varphi'$ on $M$ such that $\varphi'$ is topologically equivalent to $\varphi$ in $\mathring{N}$ and is stationary in $\partial N$.  Then, the restriction flow $\varphi'|N$ can be extended to a flow $\widehat{\varphi}$ on $\widehat{N}$ by keeping all the points in $\widehat{N}-N$ fixed.  Besides, the flow $\widehat{\varphi}$ is topologically equivalent to a $C^1$ flow $\psi$ by Guti\'errez' Theorem and, as a consequence, $\varphi'|\mathring{N}$ is topologically equivalent to $\psi|h(\mathring{N})$, where $h:\widehat{N}\to \widehat{N}$ is the homeomorphism which realizes the equivalence. Therefore, \cite{Conley_Isolated_1971} ensures the existence of a basis of isolating block manifolds of $K$ for $\psi$ and, hence, for $\varphi$.
\end{proof}

The next proposition gives a topological characterization of the initial sections of the truncated unstable manifold of an isolated invariant continuum of a flow on a surface and, as a consequence, it also characterizes the topology of the $S$-initial part of the truncated unstable manifold.

\begin{proposition}
Let $\varphi:M\times\mathbb{R}\to M$ be a flow defined on a surface, $K$ be an isolated invariant continuum and $S$ an initial section of the truncated unstable manifold $W^u(K)-K$. Then, $S$ has a finite number of connected components and each one is either an interval (possibly degenerate) or a circle. Moreover, $I^u_S(K)$ is homeomorphic to a finite disjoint union of half-open rays, strips and cylinders.
\end{proposition}

\begin{proof}
 By Lemma~\ref{smooth} we can find a connected isolating block manifold $N$ of $K$. Besides, $S$ is homeomorphic to $n^-$. Hence, Lemma~\ref{general} guarantees that it has a finite number of components. Moreover, $N^o$ consists of a disjoint union of finite many circumferences and closed intervals. Then, since $n^-$ is a compact subset of this disjoint union, it must be a finite union of points, closed intervals and circumferences as we wanted to prove. Therefore, the result follows  $I_S^u(K)$ being homeomorphic to $S\times(-\infty,0]$ .
\end{proof}

\begin{theorem}\label{polyhedron}
Let $K$ be an isolated invariant continuum of a flow on a surface. Then, $K$ has the shape of a finite polyhedron. Moreover, if $N$ is a connected isolating block manifold of $K$,
\[
\beta_1(K)\leq\beta_1(N)
\] 
\end{theorem}

\begin{proof}
Let $N$ be a connected isolating block manifold of $K$. By Alexander duality  
\[
\check{H}^2(N,K)\cong H_0(N-K,\partial N),
\]
and the latter group must be zero since, if not, there would be a component $U$ of $N-K$ not meeting $\partial N$, which means that, given $x\in U$, the trajectory $\gamma(x)$ must be contained in $N$ since it only can leave $N$ through $\partial N$. This fact contradicts $N$ to be an isolating neighborhood of $K$.

 Consider the long exact sequence of reduced \v Cech cohomology of the pair $(N,K)$
\[
0\to \check{H}^1(N,K)\to\check{H}^1(N)\to\check{H}^1(K)\to\check{H}^2(N,K)=\{0\}
\]

Therefore, the homomorphism  $i^*:\check{H}^1(N)\to\check{H}^1(K)$ is surjective and, since $\check{H}^1(N)$ is finitely generated, being $N$ a compact manifold, so is $\check{H}^1(K)$. Thus, $K$ has the shape of a wedge of $\beta_1(K)$ circumferences by Corollary~\ref{jaime2} and $\beta_1(K)\leq\beta_1(N)$.
\end{proof}

\begin{corollary}\label{disk}
Let $K$ be an isolated invariant continuum of a flow on a surface. Suppose that $K$ admits an isolating block which is a disk, then $K$ has trivial shape and  contains a fixed point.
\end{corollary}

\begin{proof}
Since $\beta_1(N)=0$, Theorem~\ref{polyhedron} guarantees that $\beta_1(K)=0$ and, hence, Theorem~\ref{jaime2} ensures that $K$ has trivial shape. Let us see that $K$ must contain a fixed point. Since $K$ admits an isolating block $N$ which is a disk, this disk can be embedded into $\mathbb{R}^2$ and, by the arguments presented in the proof of Lemma~\ref{smooth}, we may assume, without loss of generality, that the flow restricted to $\mathring{N}$ can be extended to a $C^1$ flow on the whole plane. This fact allows us to use Poincar\'e-Bendixson Theorem. Choose a point $x\in K$, hence $\emptyset\neq\omega(x)\subset K$ and either it contains a fixed point or it is a limit cycle. If $\omega(x)$ is a limit cycle, it must decompose $\mathbb{R}^2$ into two connected components, and, since $\mathring{N}$ is an open disk, the bounded component $U$ must be contained in $\mathring{N}$. Thus, $\overline{U}$ is an invariant disk contained in $\mathring{N}$ and, hence, in $K$, and the Brouwer fixed point theorem combined with the compact character of $\overline{U}$ ensure that $K$ must contain a fixed point.
\end{proof}

\begin{remark}
Theorem~\ref{polyhedron} does not hold for flows on higher-dimensional manifolds. For instance, consider on $\mathbb{R}^3$ the vector field
\[
X(x,y,z)=\Phi(x,y,z)\vec{e}_3,
\]
where $\vec{e}_3=(0,0,1)$ and $\Phi:\mathbb{R}^3\to\mathbb{R}$ is a $C^\infty$ function which takes the value $0$ exactly in those points which belong to the subset
\[
H=\bigcup_{n\in\mathbb{N}}\left\{(x,y,z)\in\mathbb{R}^3\; \vrule\; \left(x-\frac{1}{n}\right)^2+y^2=\frac{1}{n^2},\; z=0\right\},
\]
and it takes the value $1$ outside a neighborhood of $H$. The flow induced by $X$ is depicted in figure~\ref{fig:1} and it has the set $H$, which is known as the Hawaaian earring, as an isolated invariant set. It is clear that $H$ admits an isolating block which is a ball but, in spite of it, $\beta_1(H)=\infty$. In particular, $H$ does not have polyhedral shape.
\begin{figure}[h]
\center
\includegraphics[scale=.7]{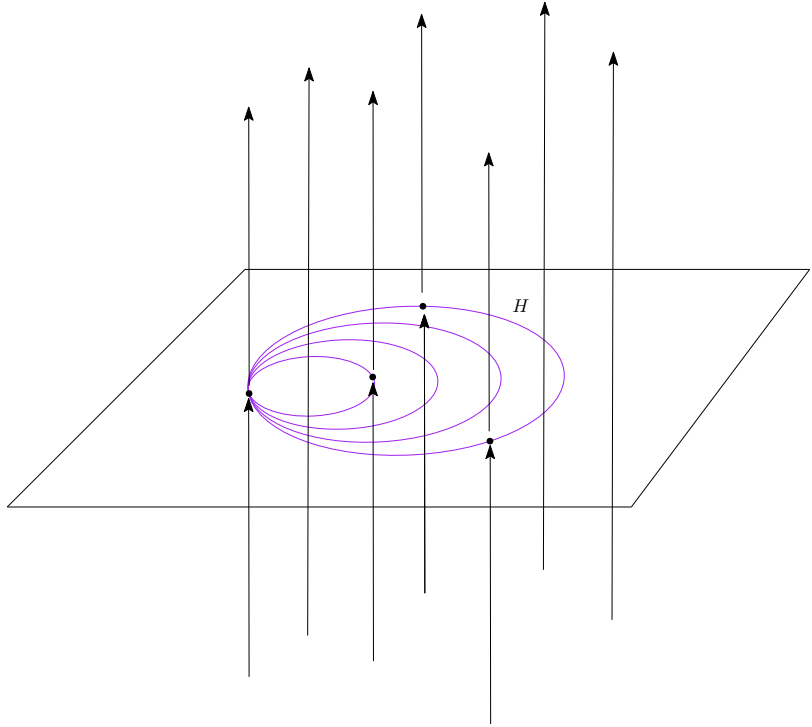}
\caption{Flow on $\mathbb{R}^3$ having the Hawaaian earring as an isolated invariant set.}
\label{fig:1}
\end{figure}

This example is a particular instance of a general result from \cite{giraldosome2001} which states that any finite dimensional compactum can be an isolated invariant set of a flow on some $\mathbb{R}^n$. This example also shows  that in higher-dimensional manifolds, given a connected isolating block manifold $N$ of an isolated invariant continuum $K$, $\beta_1(K)$ may be greater than $\beta_1(N)$. In \cite{gabites} some conditions involving $\beta_1(N)$ are used to find lower bounds of $\beta_1(K)$ for flows on 3-manifolds.
\end{remark}

\section{Regular isolating blocks and the Conley index}\label{sec:2}
In this section we will see that the knowledge of the first Betti number of an isolated invariant continuum of a flow on a surface and the topology of an initial section of its truncated unstable manifold allow us to compute its Conley index, extending in this way a result of \cite{bargeunstable2014} about planar isolated invariant continua. For this purpose we will make use of a special kind of isolating blocks, the so-called \emph{regular isolating blocks}. This kind of blocks was first introduced and studied by Easton in \cite{Easton} and subsequently studied by Gierzkiewicz and W\'ojcik \cite{Wojcik} and J.J. S\'anchez-Gabites \cite{gabites,Gabites_Aplicaciones_2009}. However, most of the known results are referred to the 3-dimensional case and the more general results, which appear in \cite{Wojcik}, do not apply to the kind of isolating blocks considered here since we are dealing with a more restrictive definition of isolating block. We will dedicate part of this section to fill this gap and prove that isolated invariant continua of flows on surfaces admit a basis of regular isolating blocks.

\begin{definition}\label{regular block}
A connected isolating block manifold $N$ is said to be \emph{regular} if the inlcusion $i:K\hookrightarrow N$ is a shape equivalence.
\end{definition}

\begin{remark}
Notice that the condition for an isolating block to be regular in Definition~\ref{regular block} differs from the one introduced and studied in \cite{Easton, Wojcik}. However, from the considerations made in the Introduction it follows that for connected isolating block manifolds in surfaces both definitions agree. In addition, it also follows that all regular isolating blocks of the same isolated invariant continuum must be homeomorphic. This facts also hold in 3-manifolds \cite{gabites,Gabites_Aplicaciones_2009}.
\end{remark}

\begin{theorem}\label{blockbasis}
If $K$ is an isolated invariant continuum of a flow on a surface, it admits a basis of regular isolating blocks.
\end{theorem}

\begin{proof}
 Let $N$ be a connected isolating block manifold of $K$. From the proof of Theorem~\ref{polyhedron} we have that the sequence
\[
0\to \check{H}^1(N,K)\to\check{H}^1(N)\to\check{H}^1(K)\to 0,
\]
is exact and, as a consequence, from Theorem~\ref{jaime1}, the obstruction for $N$ to be a regular block is the existence of non-trivial elements in $\check{H}^1(N,K)$. On the other hand, as we have seen in the proof of Lemma~\ref{general}, $\check{H}^1(N,K)\cong\check{H}^1(N^o,n^-)$ and, by Alexander duality, we get 
\[
\check{H}^1(N^o,n^-)\cong H_0(N^o-n^-,\partial N^o).
\]
Notice that $H_0(N^o-n^-,\partial N^o)$ is finitely generated. We will construct the desired block from $N$ by cutting from it the leftover information in the following way:

 Assume that $C$ is a circular component of $N^o$ not contained in $n^-$. Each component of $C-n^-$ represents a generator of $H_0(N^o-n^-,\partial N^o)$ since it does not contain points of $\partial N^o$. Choose a point $x_0\in (C-n^-)$ and a compact and connected neighborhood $U$ of $x_0$ in $C$ disjoint from $n^-$. Notice that $U$, being a proper nondegerate subcontinuum of the circle must be homeomorphic to the unit interval $[0,1]$. Thus, Lemma~\ref{general} guarantees that the set
\[
W=\bigcup_{x\in U}x[t^i(x),0],
\] 
is homeomorphic to the unit square $[0,1]\times [0,1] $ via a homeomorphism $h:W\to[0,1]\times [0,1]$ which carries each segment of trajectory $x[t^i(x),0]$ to $\{g(x)\}\times [0,1]$, where $g:U\to[0,1]$ is a homeomorphism. Now we will perform the following operation: choose in $[0,1]\times [0,1]$ the parabolic segments $\alpha$ and $\beta$ depicted in figure~\ref{fig:2} and let $R$ be the open region between these curves in $[0,1]\times[0,1]$. Then, if we consider
 \begin{figure}
\center
\includegraphics[scale=1]{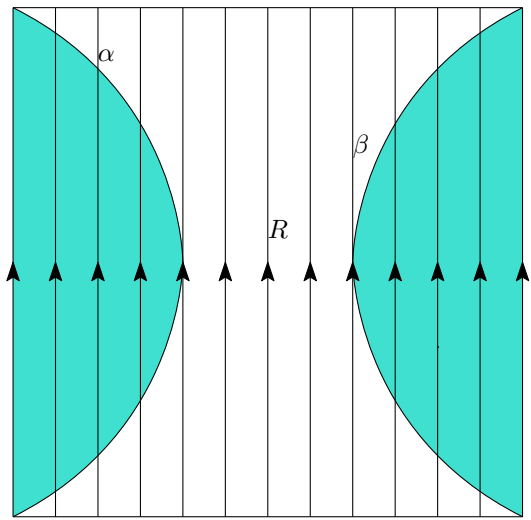}
\caption{The curves $\alpha$ and $\beta$ and the region $R$ in $[0,1]\times [0,1]$.}
\label{fig:2}
\end{figure}
$N_{(1)}=N-h^{-1}(R)$, it is clear from the construction that it is a connected isolating block manifold. Notice that this operation keeps $n^-$ unaltered. Moreover, the number of boundary components has been reduced by $1$ since the component $C$ has been joined with a component of $N^i$, which lies in a different component of $\partial N$. As a consequence, $C$ becomes an interval, say $J$, and $J-n^-_{(1)}$  has one more component than $C-n^-$. However, $J$ must contain two points of $\partial N^o$, each one lying in a different component of $J-n^-_{(1)}$ and, thus, the homology group $H_0(N_{(1)}^o-n_{(1)}^-,\partial N_{(1)}^o)$ has exactly one generator less than $H_0(N^o-n^-,\partial N^o)$. After performing this operation to each circular component of $N^o$ not contained in $n^-$ we obtain a connected isolating block manifold $N_{(r)}$ such that, all the circular components of $N^o_{(r)}$ are contained in $n^-_{(r)}$.

We will denote $N_{(r)}$ by $N$ since it should not lead to confusion. Choose a component $J$ of $N^o$ which contains more than one component of $n^-$. Then, $J$ must be an interval. Thus, each component of $J-n^-$ not containing one of the endpoints represents a generator of $H_0(N^o-n^-,\partial N^o)$. Choose an orientation in $J$ and let $n^-_1$ and $n^-_2$ be the first and the second components of $n^-$ appeared regarding the chosen orientation. Choose a point in the interval $J$ lying between $n^-_1$ and $n^-_2$ and perform the previously described operation. We obtain in this way a new isolating block manifold $N_{(1)}$ in which the component $J$ has been splitted into two disjoint exit intervals, one of them containing $n^-_1$ and the other containing remaining components of $n^-$ which were contained in the original $J$. Notice that $N_{(1)}$ is also connected since, if not, $K$ and one of the chosen components of $n^-$ should lie in different components of $N_{(1)}$ and this cannot happen. If we perform this operation until we separate all the components of $n^-$ (i.e. a finite number of times) we get the desired block.  
\end{proof}

\begin{definition}
A non-empty continuum $K$ contained in a surface is said to be \emph{orientable} if it admits a basis of neighborhoods comprised of orientable surfaces. Otherwise $K$ is said to be \emph{nonorientable}.
\end{definition}

\begin{remark}
From the proof of Lemma~\ref{smooth} it follows that any continuum in the interior of a surface has a neighborhood basis comprised of compact and connected 2-manifolds with boundary. Combining this with the fact that an orientable 2-manifold cannot contain a nonorientable one it follows
\begin{enumerate} 
\item[i)] Every continuum contained in the interior of an orientable surface must be orientable. 

\item[ii)] An orientable continuum $K$ cannot possess a basis of neighborhoods comprised of nonorientable manifolds.

\item[iii)] A nonorientable continuum $K$ must admit a basis of neighborhoods comprised of nonorientable surfaces.
\end{enumerate}   
\end{remark}

However as the next example points out, nonorientable surfaces contain both orientable and nonorientable compact subsets.

\begin{example}
Consider $M$ as the surface obtained as a connected sum of the torus $S^1\times S^1$ with the Klein bottle $\mathcal{K}$ (which is homeomorphic to a connected sum of four projective planes \cite{Massey}). In this surface we can find two copies of $S^1\vee S^1$ as the 1-skeleton of the torus and the Klein bottle summands respectively. It is clear that the one contained in the torus summand is orientable while the other is not.   
\end{example}

Now we are ready to prove the main result of this section.

\begin{theorem}\label{main}
Suppose $K$ is an isolated invariant continuum of a flow $\varphi :
M\times\mathbb{R} \rightarrow M$ defined on a surface. Let $u$ be the number of components of an initial section $S$ of the truncated unstable manifold and  $u_c$ the number of contractible components of $S$. Then 
\begin{enumerate}
\item[i)] If $K$ is neither an attractor, nor a repeller the Conley index of $K$ is the pointed homotopy type of $\left( \bigvee_{i=1,\ldots,k}S_{i}^{1},\ast \right) ,$
where $k=\beta_1(K)+u_{c}-1$ and $S_{i}^{1}$ is a pointed $1-$sphere based on $\ast $
for $i=1,\ldots,k.$

\item[ii)] If $K$ is an attractor, $u=0$ and its Conley index is the pointed
homotopy type of $\left( \bigvee_{i=1,\ldots,\beta_1(K)}S_{i}^{1}\cup \{\bullet
\},\bullet \right) ,$ where the $S_{i}^{1}$ are pointed $1-$spheres based on 
$\ast $ and $\bullet $ denotes a point not belonging to $\bigvee_{i=1,\ldots,\beta_1(K)}S_{i}^{1}.$

\item[iii)] If $K$ is a repeller:
\begin{enumerate}
\item If $K$ is orientable its Conley index is the pointed
homotopy type of $\left( \Sigma_g\bigvee \left(
\bigvee_{i=1,\ldots,u-1}S_{i}^{1}\right) ,\ast \right) $, where $\Sigma_g$ is a closed orientable surface of genus $g=\frac{1+\beta_1(K)-u}{2}$. The surface $\Sigma_g$  and all the $S_{i}^{1}$ are pointed and based on $\ast.$
\item If $K$ is nonorientable its Conley index is the pointed
homotopy type of $\left(N_g\bigvee \left(\bigvee_{i=1,\ldots,u-1}S_{i}^{1}\right) ,\ast \right)$, where $N_g$ is a closed nonorientable surface of genus $g=1+\beta_1(K)-u$. The surface $N_g$  and all the $S_{i}^{1}$ are pointed and based on $\ast.$
\end{enumerate}
\end{enumerate}
\end{theorem}
\begin{proof}

Let $N$ be a regular isolating block of $K$. Then, given an initial section $S$ of the truncated unstable manifold $W^u(K)-K$, $S$ is homotopy equivalent to $N^o$. Indeed, since the inclusion $i:K\hookrightarrow N$ is a shape equivalence, the cohomology groups $\check{H}^k(N,K)=\{0\}$. But, as we have seen before $\check{H}^k(N,K)\cong\check{H}^k(N^o,n^-)$ and, hence, $i:n^-\hookrightarrow N^o$ induces isomorphisms in \v Cech cohomology. It easily follows that $n^-$ and $N^o$ have the same homotopy type and the claim follows $n^-$ being homeomorphic to $S$. 

From this observation we get that $N^o$ has $u_c$ components which are intervals and $u-u_c$ circular components. 

Suppose that $K$ is neither an attractor nor a repeller and let $N$ be a regular isolating block of $K$. The block $N$ is a compact 2-manifold with boundary and, since it has the same shape as $K$ it must have the homotopy type of a wedge of $\beta_1(K)$ circumferences. Collapsing to a point an interval component of $N^o$ does not change the homotopy type of $N$. Therefore, the topological space obtained by collapsing all the interval components to a single point is pointed homotopy equivalent to the wedge sum of $N$ with $u_c-1$ copies of $S^1$. On the other hand, collapsing a circular component $C$ of $N^o$ produces the same effect on $N$ as capping the boundary component $C$ with a disk. Then, the topological space obtained by collapsing to a point all the circle components is pointed homotopy equivalent to a wedge sum of $(u-u_c-1)$ circumferences with the manifold obtained after capping $(u-u_c)$ boundary components with disks. Thus, since $N^o$ is neither empty nor the whole $\partial N$  the Conley index of $K$ must be the pointed homotopy type of a wedge sum of a compact and connected 2-manifold with boundary with some circumferences. Hence, it must be pointed homotopy equivalent to a wedge of circumferences. To determine the number of circumferences on the wedge we compute the Euler characteristic of $h(K)$. Since $\chi(h(K))$ agrees with $\chi(N,N^o)$ and $N^o$ is a union of $u_c$ intervals and $u-u_c$ circumferences it follows 
\[
\chi(h(K))=\chi(N)-\chi(N^o)=1-\beta_1(N)-u_c,
\]  
and, hence, $\rank CH^1(K)=\beta_1(N)+u_c-1$. This proves i).

If $K$ is an attractor it admits a positively invariant isolating neighborhood and, hence, $u=0$. Thus, if $N$ is a regular isolating block it must have empty exit set. As a consequence, the effect of collapsing its exit set $N^o$ to a point is the same as making the disjoint union of $N$ with a singleton not contained in $N$. This proves ii).

Suppose that $K$ is a repeller. Then, given a regular isolating block $N$ of $K$, $N^o$ must be the whole boundary $\partial N$ which is comprised of $u$ connected components. The space obtained after collapsing the whole boundary of $N$ to a point is pointed  homotopy equivalent to the wedge sum of $u-1$ circumferences with the surface obtained after capping all the boundary components of $N$ with disks. This surface is orientable if and only if $K$ is orientable. Indeed, if $K$ is orientable it admits a basis of neighborhoods comprised of orientable 2-manifolds with boundary. As a consequence, $K$ admits an orientable regular isolating block. If $K$ is nonorientable the same argument shows that $K$ admits a nonorientable regular block. Let us compute the genus of $S_g$, the closed surface obtained after capping with a disk each boundary component of $\partial N$.  Since, $\partial N$ has exactly $u$ components, using the fact that
\[
\chi(A\cup B)=\chi(A)+\chi(B)-\chi(A\cap B)
\]
it easily follows that
\[
\chi(S_g)=1-\beta_1(N)+u.
\]
On the other hand,
\[
\chi(S_g)=
\begin{cases}
2-2g & \mbox{if $S_g$ is orientable}\\
2-g & \mbox{otherwise}
\end{cases}
\] 
This proves iii).
\end{proof}

\begin{remark}
Notice that in the item iii) of Theorem~\ref{main} the genus of the surface which appears as a direct summand must be less than or equal to than the genus of the phase space $M$. This can be easily seen using the Mayer-Vietoris sequence. 
\end{remark}

\section{The cohomology index}\label{sec:3}

The aim of this section is to study the cohomology index of an isolated invariant continuum of a flow on a surface and its relations with the Conley index. Since cohomology groups are easier to compute than homotopy type it is interesting to study to what extent the cohomology index determines the Conley index.  

\begin{example}
Let $M$ be an orientable surface of genus greater than or equal to $1$ and consider two  flows $\varphi$ and $\varphi'$ on $M$ having isolated invariant sets $K_1$ and $K_2$ respectively whose local dynamics are depicted in figures \ref{fig:3} and \ref{fig:4}. The Conley indices of $K_1$ and $K_2$ are the pointed homotopy type of $\left(S^2\vee S^1_1\vee S^1_2,*\right)$ and $\left(S^1\times S^1,*\right)$. Then, their cohomology indices agree being

\begin{figure}[h]
\center
\includegraphics[scale=.8]{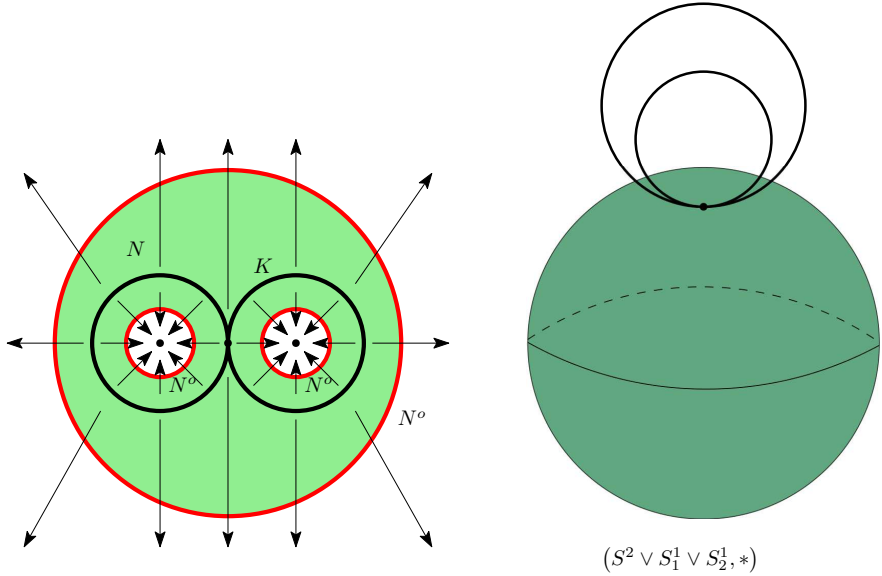}
\caption{Flow having $S^1\vee S^1$ as a repeller whose Conley index is the pointed homotopy type of $\left(S^2\vee S^1_1\vee S^1_2,*\right)$.}
\label{fig:3}
\end{figure}

\begin{figure}
\center
\includegraphics[scale=.7]{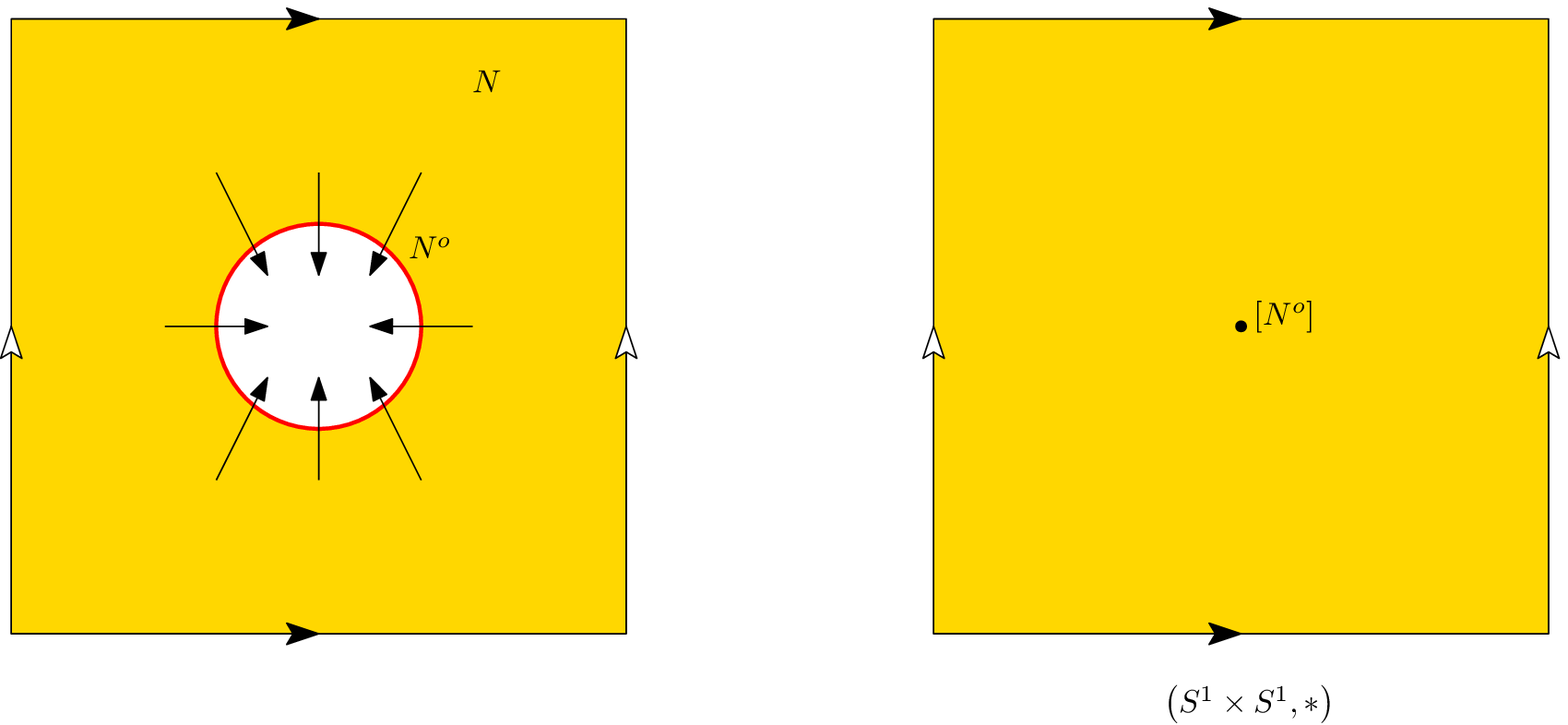}
\caption{Flow having $S^1\vee S^1$ as a repeller whose Conley index is the pointed homotopy type of $\left(S^1\times S^1,*\right)$.}
\label{fig:4}
\end{figure}

\[
CH^i(K_j)=\begin{cases}
\mathbb{Z}_2\oplus\mathbb{Z}_2 & \mbox{if}\quad i=1 \\
\mathbb{Z}_2  & \mbox{if}\quad i=2\\
0  & \mbox{otherwise}
\end{cases}
\]

However, these spaces are not homotopy equivalent. This can be seen using the ring structure of $CH^*(K_1)$ and $CH^*(K_2)$. As rings
\begin{eqnarray*}
CH^*(K_1)&\cong & \widetilde{H}^*(S^2)\oplus\widetilde{H}^*(S^1_1)\oplus\widetilde{H}^*(S^1_2)\\
CH^*(K_2)&\cong & \widetilde{H}^*(S^1\times S^1).
\end{eqnarray*}
Let $\sigma_1$, $\sigma_2$ be elements of $CH^1(K_1)$. Then $\sigma_i=a_i\gamma_1+b_i\gamma_2$, where $\gamma_i$ is the generator of $H^1(S^1_i)$, $i=1,2$. As a consequence, $\sigma_1\smile\sigma_j=0$ since $\gamma_1\smile\gamma_2=0$ by the direct sum structure of $CH^*(K_1)$ and $\gamma_i\smile\gamma_i\in H^2(S^1_i)=0$, $i=1,2$.

On the other hand, if $\alpha$, $\beta$ are the standard generators of $CH^1(K_2)$,  $\alpha\smile\beta$ generates $CH^2(K_2)\cong\mathbb{Z}_2$. Therefore, the rings $CH^*(K_1)$ and $CH^*(K_2)$ are not isomorphic and $h(K_1)\neq h(K_2)$.
\end{example}

The previous example shows that the knowledge of the groups which conform the cohomology index is not enough to know the Conley index. We will see that in spite of it, the cohomology ring $CH^*(K)$ determines the Conley index.

Given a topological space $M$ with $H^2(M)=\mathbb{Z}_2$ it is possible to define a bilinear form
\[
I:H^1(M)\times H^1(M)\to \mathbb{Z}_2,
\]
given by $I(\alpha_1,\alpha_2)=\alpha_1\smile\alpha_2$. This form determines the cohomology ring $H^*(M)$ when $M$ is a closed surface. The rank $I$ is defined as the rank of any matrix representing $I$. This number is well defined since two matrices representing $I$ must be congruent.

\begin{theorem}\label{ring}
Suppose that $K$ is an isolated invariant continum of a flow on a surface. Then, the cohomology ring $CH^*(K)$ determines its Conley index. In particular,
\begin{enumerate}
\item[i)] If $CH^0(K)=CH^2(K)=\{0\}$, then $K$ is neither an attractor nor a repeller and its Conley index is the pointed homotopy type of $\left( \bigvee_{i=1,\ldots,s}S_{i}^{1},\ast \right)$, where $s$ agrees with $\rank CH^1(K)$.

\item[ii)] If $CH^0(K)\neq\{0\}$ then $K$ is an attractor and its Conley index is the pointed homotopy type of $\left( \bigvee_{i=1,\ldots,s}S_{i}^{1}\cup \{\bullet\},\bullet \right)$ where $s$ agrees with $\rank CH^1(K)$. In particular, $K$ has the shape of $s$ circumferences.

\item[iii)] If $CH^2(K)\neq\{0\}$ then $K$ is a repeller and:
\begin{enumerate}
\item If $\alpha\smile\alpha=0$ for each $\alpha\in CH^1(K)$ the Conley index of $K$ is the pointed homotopy type of $\left( \Sigma_g\bigvee \left(
\bigvee_{i=1,\ldots,r}S_{i}^{1}\right) ,\ast \right)$, where $g=\frac{\rank I}{2}$ and $r=\rank CH^1(K)-2g$.  

\item If there exists $\alpha\in CH^1(K)$ such that $\alpha\smile\alpha\neq 0$ the Conley index of $K$ is the pointed homotopy type of $\left( N_g\bigvee \left(
\bigvee_{i=1,\ldots,r}S_{i}^{1}\right) ,\ast \right)$, where $g=\rank I$ and $r=\rank CH^1(K)-g$.
\end{enumerate}
In both cases the number of components of an initial section $S$ of the truncated unstable manifold is $r+1$ and $K$ has the shape of $\rank CH^1(K)$ circumferences.
\end{enumerate}
\end{theorem}

\begin{proof}
Suppose that $CH^0(K)=CH^2(K)=\{0\}$, then, $h(K)$ must be connected and it cannot  contain any closed surface as a wedge summand. Thus, Theorem~\ref{main} ensures that it cannot be an attractor or a repeller and $h(K)$ must be the homotopy type of a wedge of circumferences. It is clear that the number of circumferences in the wedge is determined by $\rank CH^1(K)$. This proves i).

Let us assume that $CH^0(K)\neq\{0\}$. Then $h(K)$ is not connected and by Theorem~\ref{main} it must be an attractor. Moreover, $h(K)$ must have the homotopy type of the union of a wedge of circumferences and an exterior point. As before $\rank CH^1(K)$ determines the number of circumferences in the wedge.

To prove iii) assume that $CH^2(K)\neq\{0\}$, then Theorem~\ref{main} guarantees that $K$ is a repeller. Moreover, $h(K)$ must contain a closed connected surface as a wedge summand. This surface is orientable (and hence $K$ is orientable) if and only if, given any element $\alpha\in CH^1(K)$, $\alpha\smile\alpha=0$. This is a straightforward consequence of the cohomology ring structure of closed surfaces (See \cite{HatcherAT}). 

Suppose that $K$ is orientable. Then $h(K)$ is the pointed homotopy type of $\left(\Sigma_g\bigvee \left(
\bigvee_{i=1,\ldots,r}S_{i}^{1}\right) ,\ast \right)$. Let us show that $g$ is exactly $\rank I/2$. By \cite{HatcherAT} we have that
\begin{equation}
CH^*(K)\cong \widetilde{H}^*(\Sigma_g)\oplus\left(\bigoplus_{i=1}^r \widetilde{H}^*(S^1_i)\right),
\label{eq:1}
\end{equation}
as rings. Choose the basis $\{\alpha_1,\ldots,\alpha_g,\beta_1,\ldots\beta_g,\gamma_1,\ldots,\gamma_r\}$ of $CH^1(K)$ where $\{\alpha_1,\ldots,\alpha_g,\beta_1,\ldots,\beta_g\}$ is the standard basis of $H^1(\Sigma_g)$ and each $\gamma_i$ is  the generator of $H^1(S^1_i)$ for each $i$. Let $\sigma$ be the generator of $CH^2(K)\cong\mathbb{Z}_2$, then

\[
\alpha_i\smile\beta_j= 
\begin{cases}
\sigma & \mbox{if}\quad i=j\\
0\quad \mbox{if}\quad i\neq j
\end{cases}
\]
and $\alpha_i\smile\alpha_j=0$, $\beta_i\smile\beta_j=0$ for each $i,j$. Besides, \eqref{eq:1} ensures that $\gamma_i\smile\omega=0$ for each $i=1,\ldots,r$ and each $\omega\in CH^1(K)$. Therefore, the matrix associated to the bilinear form $I$ with respect to the chosen basis takes the form
\[
  \begin{pmatrix}
   O_g & \vrule &  I_g &\vrule &  
    \\  \cline{1-3}
    I_g  &\vrule & O_g & \vrule & O_{s\times r} \\\cline{1-3} 
    & O_{r\times 2g} & &\vrule &
   \end{pmatrix}
\]
where $I_g$ denotes the order $g$ identity matrix, $O$ denotes the zero matrix of the corresping order and $s=\rank CH^1(K)$. Hence, the rank of $I$ is $2g$ and the result follows.

Suppose that $K$ is nonorientable. In this case $h(K)$ is the pointed homotopy type of $\left(N_g\bigvee \left(
\bigvee_{i=1,\ldots,r}S_{i}^{1}\right) ,\ast \right)$. We see that $g$ agrees with the rank of $I$. Consider the basis $\{a_1,\ldots,a_g,\gamma_1,\ldots,\gamma_r\}$  of $CH^1(K)$ where $\{a_1,\ldots,a_g,\}$ is the standard basis of $H^1(N_g)$ and each $\gamma_i$ is  the generator of $H^1(S^1_i)$ for each $i$. Let $\sigma$ be the generator of $CH^2(K)\cong\mathbb{Z}_2$, then 
\[
a_i\smile a_j= 
\begin{cases}
\sigma & \mbox{if}\quad i=j\\
0\quad \mbox{if}\quad i\neq j
\end{cases}
\]
and, reasoning as before, $\gamma_i\smile\omega=0$ for each $i=1,\ldots r$ and $\omega\in CH^1(K)$. Therefore, the matrix associated to the bilinear form $I$ with respect to the chosen basis takes the form
\[
  \begin{pmatrix}
   I_g & \vrule  &  O_{g\times r}
    \\ \cline{1-3}
 O_{r\times g} &  \vrule & O_{r\times r}
   \end{pmatrix}
\]
Thus, the rank of $I$ is $g$ and the result follows.

Notice that from this discussion it also follows that the cohomology ring $CH^*(K)$ determines $h(K)$ as we wanted to prove.
\end{proof}

\section{Applications}\label{sec:4}

In this section we will show some applications of the previous results. For instance, we will relate the property of being non-saddle with the structure of the truncated unstable manifold, we will study properties of continuations of isolated invariant continua and we will characterize those isolated invariant continua which do not have fixed points.    

The next result shows a duality property of those  isolated invariant continua which are neither attractors nor repellers for flows on surfaces

\begin{proposition}
Suppose that $K$ is an isolated invariant continuum of a flow on a surface. Then, the number of contractible components of an initial section $S$ of the truncated unstable manifold of $K$, $u_c$, agrees with the number of contractible components $s_c$ of a final section $S^*$ of the truncated stable manifold. As a consequence, if $K$ is neither an attractor nor a repeller, the Conley index $h(K)$ agrees with the Conley index for the reverse flow $h^*(K)$. 
\end{proposition}

\begin{proof}
Consider a regular isolating block $N$ of $K$. As we have seen in the proof of Theorem~\ref{main}, $N^o$ posseses exactly $u_c$ interval components and, working with the reverse flow, it also follows that $N^i$ has exactly $s_c$ interval components. Since $\partial N$ is a disjoint union of circumferences, it is clear that the number of components of $N^o$ and $N^i$ contained in a component $C$ of $\partial N$ not contained neither in $n^-$ nor in $n^+$ must be the same and, hence, $u_c=s_c$. The remaining part of the statement follows straightforward from Theorem~\ref{main}. 
\end{proof}

In the next result we will see that the vanishing of $u_c$ is related to the dynamical property of being \emph{non-saddle} introduced by Bhatia in \cite{Bhatia}.
\begin{definition}
A compact invariant set $K$ is said to be \emph{saddle} if it admits a neighborhood $U$ such that every neighborhood $V$ of $K$ contains a point $x\in V$ with $x[0,+\infty)\nsubseteq U$ and $x(-\infty,0]\nsubseteq U$. Otherwise we say that $K$ is \emph{non-saddle}. 
\end{definition}

For instance, attractors, repellers and unstable attractors with no external explosions (see \cite{Athanassopoulos_Explosions_2003,Morón_Topology_2007,GabitesTrans}) are non-saddle sets. For more information about these sets see \cite{giraldosome2001, Giraldo_Topological_2009}. 

\begin{proposition}\label{non-saddle}
An isolated invariant continuum $K$ of a flow on a surface is non-saddle if and only if $u_c=0$.
\end{proposition}

\begin{proof}
Suppose, arguing by contradiction, that $K$ is non-saddle and $u_c\neq 0$. Let $N$ be a regular isolating block of $K$. Since $u_c\neq 0$, $N^o$ must have at least one component $J$ which is an interval. Besides, $J$ must contain in its interior a component of $n^-$. Let $(x_n)$ be a sequence in $J-n^-$ convergent to $x\in n^-$. The trajectory of each $x_n$ must leave $N$ in the past and in the future but, since $x_n\to x\in n^-$, fixed any neighborhood $U$ of $K$, there exists $n$ such that the trajectory of $x_n$ meets $U$ before it leaves $N$ in the past. This is in contradiction with $K$ being non-saddle.     

Conversely, assume that $u_c=0$. Then, given a regular isolating block $N$ of $K$, $N^o$ must agree with $n^-$. We will see that given any $x\in N$, either $\gamma^+(x)$ or $\gamma^-(x)$ is contained in $N$. Suppose, arguing by contradiction, that there exists a point $x$ whose trajectory leaves $N$ in the past and in the future. Thus, the exit time function $t^o$ is defined in $x$ and, $xt^o(x)\in N^o=n^-$. As a consequence, $x\in N^-$ which is in contradiction with the trajectory of $x$ leaving $N$ in the past. Therefore, $K$ is non-saddle since $K$ admits a basis of neighborhoods comprised of regular isolated blocks.
\end{proof}

\begin{remark}
In \cite{giraldosome2001} it was proved that isolated non-saddle sets (possibly not connected) in manifolds have the shape of finite polyhedra. Then, they have a finite number of components, each one of them isolated and non-saddle. As a consequence, Proposition~\ref{non-saddle} also holds if $K$ has a finite number of components. 
\end{remark}

Now we are going to study some questions about continuation of isolated invariant continua. The theory of continuation plays a central role in the Conley index theory. We recommend \cite{Salamon, Conley_mono, Giraldo_Singular_2009} for information about the basic facts of this notion. In figure~\ref{fig:8} we show an example from \cite{Giraldo_Topological_2009}, in which it is shown that some dynamical and topological properties of the original isolated invariant continuum are not preserved by its continuations. For instance, connectivity, shape or non-saddleness are some properties not preserved by continuation. We will see that in spite of this fact, for flows on surfaces we can have a good understanding of how continuations work.   

\begin{figure}[h]
\center
\includegraphics[scale=1]{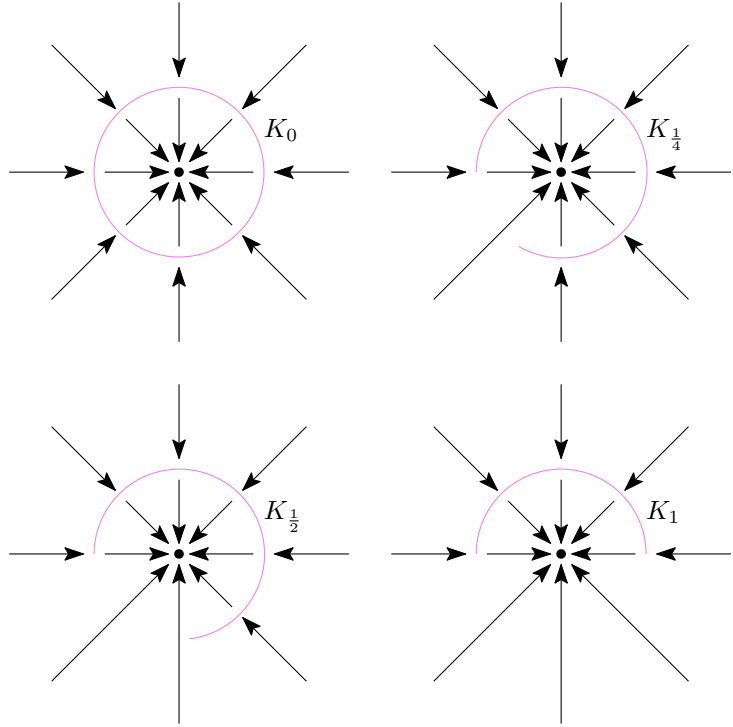}
\caption{Continuation of an isolated non-saddle circumference by a family of saddle sets with the shape of a point.}
\label{fig:8}
\end{figure}

The next result shows that even if the continuation of an isolated invariant continuum is not connected, for small values of the parameter each one of its components is shape dominated by the  the original continuum. This means that small perturbations of a flow cannot increase, in a certain sense, the topological complexity of isolated invariant continua. 

\begin{theorem}\label{contshape}
Let $(\varphi_\lambda)_{\lambda\in I}$ be a parametrized family of flows defined on a surface $M$ and $K_0$ an isolated invariant continuum for $\varphi_0$. Suppose that the family $(K_\lambda)_{\lambda\in I}$ is a continuation of $K_0$. Then, there exists $\lambda_0\leq 1$ such that 
\[
\beta_1(K_\lambda)\leq\beta_1(K_0),\quad\mbox{if} \quad \lambda\leq\lambda_0.
\]
In particular, if $\lambda\leq\lambda_0$ and $K_\lambda^\alpha$ is a component of $K_\lambda$ then, $\Sh(K_0)\geq \Sh(K^\alpha_\lambda)$.  
\end{theorem}

\begin{proof}
Let $N$ be a regular isolating block of $K_0$. Then, by \cite{Salamon} there exists $\lambda_0\leq 1$ such that $N$ is an isolating neighborhood of $K_\lambda$ for $\lambda\leq\lambda_0$. Hence, reasoning as in the proof of Theorem~\ref{polyhedron} we get that 
\[
i^*_\lambda:\check{H}^1(N)\to\check{H}^1(K_\lambda)
\]
is surjective. Therefore, $\beta_1(K_\lambda)\leq\beta_1(N)$ and the result follows since $N$ is a regular block of $K_0$.
\end{proof}

We will study continuations of isolated invariant continua regarding their dynamical nature. For this purpose we will make use of the next proposition.

\begin{proposition}\label{equ}
Let $(\varphi_\lambda)_{\lambda\in I}$ be a parametrized family of flows defined on a surface $M$ and $K_0$ an isolated invariant continuum for $\varphi_0$ which is neither an attractor nor a repeller. Suppose that the family $(K_\lambda)_{\lambda\in I}$ is a continuation of $K_0$ such that $K_\lambda$ consists of a finite number of connected components. Then
\[
\left(\beta_1(K_0)-\beta_1(K_\lambda)\right)+\left(u_c-u^\lambda_c\right)=1-n_\lambda,
\]
where $u^\lambda_c$ is the number of contractible components of an initial section of the truncated unstable manifold $W^u(K_\lambda)-K_\lambda$ and $n_\lambda$ is the number of components of $K_\lambda$.
\end{proposition}

\begin{proof}
Since $K_\lambda$ has a finite number of components $K_\lambda^ 1,\ldots,K_\lambda^ {n_\lambda}$, all of them are isolated and
\begin{equation}
h(K_\lambda)=\bigvee_{i=1}^{n_{\lambda}} h(K_\lambda^i)  \label{wedge}
\end{equation}
Moreover, $h(K_0)=h(K_\lambda)$ and, hence, $h(K_\lambda)$ is the pointed homotopy type of a wedge of $\beta_1(K_0)+u_c-1$ circumferences. The result follows from \eqref{wedge} and Theorem~\ref{main}.
\end{proof}

\begin{theorem}\label{cont2}
Let $(\varphi_\lambda)_{\lambda\in I}$ be a parametrized family of flows defined on a surface $M$ and $K_0$ an isolated invariant continuum for $\varphi_0$. Suppose that the family $(K_\lambda)_{\lambda\in I}$ is a continuation of $K_0$. Then,
\begin{enumerate}
\item[i)] If $K_0$ is an attractor (repeller), $K_\lambda$ has a component $K^1_\lambda$ which is also an attractor (repeller) and $\Sh(K^1_\lambda)=\Sh(K_0)$.

\item[ii)] If $K_0$ is neither an attractor nor a repeller then $K_\lambda$ is neither an attractor nor a repeller and:
\begin{enumerate}
\item If $K_0$ is saddle there exists $\lambda_0\leq 1$ such that  $K_\lambda$ is also saddle for $\lambda\leq\lambda_0$. 
\item If $K_0$ is non-saddle and $K_\lambda$ is a continuum for each $\lambda$, then  $K_\lambda$ is non-saddle  if and only if $\Sh(K_0)=\Sh(K_\lambda)$.
\end{enumerate}
\end{enumerate}
\end{theorem}
\begin{proof}
Suppose that $K_0$ is an attractor. The case of $K_0$ being a repeller is completely analogous reasoning with the reverse flow. Since $h(K_0)=h(K_\lambda)$, this means that $h(K_\lambda)$ is the pointed homotopy type of $\left(\bigvee_{i=1,\ldots,\beta_1(K_0)} S^1_i\cup\{\bullet\},\bullet\right)$. A consequence of this fact is that either $K_\lambda$ is an attractor or it is the disjoint union of an attractor $K^1_\lambda$ and an isolated invariant set $K^2_\lambda$ with trivial Conley index. Indeed, if $K_\lambda$ is connected then by Theorem~\ref{main} it must be an attractor and its first Betti number must agree with $\beta_1(K_0)$. Hence, $\Sh(K_\lambda)=\Sh(K_0)$. Suppose, on the other hand, that $K_\lambda$ is not connected. Then, given an isolating block $N$ of $K_\lambda$ it cannot be connected and it must contain a connected component with empty exit set. Let $K^1_\lambda$ be the isolated invariant set contained in that component. It follows that $K^1_\lambda$ is an attractor and, by the additive property of the Conley index, the index of $K^2_\lambda=K_\lambda-K^1_\lambda$ must be trivial. Besides, $K^1_\lambda$ must be connected, since if not it would be the disjoint union of $i>1$ attractors and its Conley index would be the homotopy type of a space with $i+1$ components. This is not possible since $h(K_\lambda^1)$ must agree with $h(K_0)$ which is the pointed homotopy type of a 2-component space. Therefore,  Theorem~\ref{main} ensures that $\beta_1(K^1_\lambda)$ agrees with $\beta_1(K_0)$ and, as a consequence, $\Sh(K^1_\lambda)=\Sh(K_0)$. 

Suppose that $K_0$ is neither an attractor, nor a repeller.  Then, since $h(K_\lambda)$ agrees with $h(K_0)$, the former must be connected and it cannot contain any closed surface as a summand. Hence, it is easy to conclude that $K_\lambda$ is neither an attractor nor a repeller. Suppose that $K_ 0$ is saddle. If $K_\lambda$ has an infinite number of components it must be saddle since non-saddle sets have the shape of finite polyhedra \cite{giraldosome2001}.  Let us assume that $K_\lambda$ has a finite number of components.  From Theorem~\ref{contshape} we get that there exists $\lambda_0\leq 1$ such that $\beta_1(K_\lambda)\leq\beta_1(K_0)$ if $\lambda\leq\lambda_0$. Then, using this and Proposition~\ref{equ} we obtain that $u^\lambda_c\geq u_c> 0$ for $\lambda\leq \lambda_0$ and, thus, $K_\lambda$ is saddle for $\lambda\leq\lambda_0$. 

Let us assume that $K_0$ is non-saddle, i.e. $u_c=0$, and $K_\lambda$ is connected for every $\lambda$. We will see that $K_\lambda$ is non-saddle if and only if $\Sh(K_\lambda)=\Sh(K_0)$. Indeed, suppose that $K_\lambda$ is also non-saddle, i.e. $u^\lambda_c=0$. Thus, Proposition~\ref{equ} ensures that $\beta_1(K_0)$ must agree with $\beta_1(K_\lambda)$. Therefore, $K_0$ and $K_\lambda$ must have the same shape. Conversely, if $\Sh(K_\lambda)=\Sh(K_0)$ then $\beta_1(K_0)=\beta_1(K_\lambda)$ and, by Proposition~\ref{equ},  $u^\lambda_c$ must be zero. Then, $K_\lambda$ is non-saddle.
\end{proof}

\begin{remark}
We would like to point out the following things regarding Theorem~\ref{cont2}:
\begin{enumerate}
\item If $K_0$ is an attractor (repeller) it was proven in \cite{SanjurjoJMAA} that for small values of $\lambda$, $K_\lambda$ is actually connected and, hence, an attractor with the shape of $K_0$ even for flows not necessarily defined on more general spaces than surfaces such as ANR's. Besides, in the argument presented here we only exploit the fact that $h(K)$ is the pointed homotopy type of a 2-component space and, thus, the result also holds for flows on more general phase spaces. This was proven in \cite{Giraldo_Singular_2009}.

\item  It was seen in \cite{Giraldo_Topological_2009} that if $(\varphi_\lambda)_{\lambda\in I}$ is a differentiable family of flows on a differentiable manifold and $K_0$ is non-saddle, then, if $K_\lambda$ is non-saddle for small values of $\lambda$, $K_0$ and $K_\lambda$ must have the same shape, i.e. the hypotheses about the connectivity of $K_\lambda$ can be dropped for small values of $\lambda$. The converse statement is known to be true for compact oriented differentiable manifolds with $H^1(M)=0$ \cite{hbarge}. Theorem~\ref{cont2} shows that this converse statement is true for any surface and for any parameter value. 
\end{enumerate}
\end{remark}

The next results are concerned with the topological characterization of those isolated invariant continua in surfaces which do not have fixed points. For this purpose we will make use of an index introduced by Srzednicki \cite{Srzednicki} which generalizes the degree of $C^1$ vector fields on $\mathbb{R}^n$. 

The index $i(\varphi,U)$ associated to a flow defined on an Euclidean Neighborhood Retract (manifolds, CW-complexes, polyhedra...) $M$ and an open subset $U\subset M$ with compact closure is an integer number which measures, in some extent, the number of fixed points $\varphi$ has in $U$. In particular, $\varphi$ must have fixed points in $U$ if $i(\varphi,U)\neq 0$. It turns out that if $N$ is an isolating neighborhood 
\[
i(\varphi,\mathring{N})=\chi(h(K)).
\]
A straightforward consequence of this fact is the next 

\begin{proposition}\label{index}
Let $K$ be an isolated invariant continuum of a flow $\varphi$ defined on a surface and $N$ an isolating neighborhood of $K$. Then,
\[
i(\varphi,\mathring{N})=1-\beta_1(K)-u_c.
\]
\end{proposition}

\begin{theorem}\label{fixed}
Suppose that $K$ is an isolated invariant continuum of a flow on a surface $M$ and that $K$ does not contain fixed points. Then, $K$ is non-saddle and it is either a limit cycle, a closed annulus bounded by two limit cycles, or a M\"obius strip bounded by a limit cycle.
\end{theorem}
\begin{proof}
Let $N$ be a regular isolating block of $K$. Since $K$ does not have fixed points, applying Proposition~\ref{index} we get
\[
0=i(\varphi,\mathring{N})=1-\beta_1(K)-u_c.
\] 
Hence, $\beta_1(K)+u_c=1$ and we have two possibilities. The first one is that $\beta_1(K)=0$ and $u_c=1$, which must be excluded since $N$ would be a disk \cite{richardson1963} and $K$ would contain a fixed point by Corollary~\ref{disk}. The remaining possibility is $\beta_1(K)=1$ and $u_c=0$. In this case $K$ would be non-saddle and it would have the shape of a circle. Moreover, $\beta_1(N)=1$ being $N$ a regular isolating block. Therefore $N$ is either an annulus or a M\"obius strip depending on its  orientability. Indeed, capping each component of $N$ with a disk we get a closed surface $\widehat{N}$ and
\[
\chi(\widehat{N})=\chi(N)+c,
\]
where $c$ is the number of boundary components of $N$. But, since $\beta_1(N)=1$ and $N$ has non-empty boundary it follows that $\chi(N)=0$ and, hence, $\chi(\widehat{N})=c>0$. If $N$ is orientable, so is $\widehat{N}$ and, hence, $c=2$ and $\widehat{N}$ must be a sphere. Therefore, $N$ is a sphere with two open disks removed, i.e., an annulus, as we wanted to prove.  On the other hand, if $N$ is nonorientable so is $\widehat{N}$ and, as a consequence, $c=1$ and $\widehat{N}$ must be a projective plane. Then, $N$ is a projective plane with an open disk removed, i.e., a M\"obius strip. 
  
 Suppose that $N$ is orientable, i.e., an annulus. Then, $N$ can be embedded in $\mathbb{R}^2$ and, by the arguments presented in the proof of Lemma~\ref{smooth}, we may assume, without loss of generality, that the flow restricted to $\mathring{N}$ can be extended to the whole $\mathbb{R}^2$. Besides, Guti\'errez' Theorem ensures that the extended flow may be assumed to be smooth and the result follows from \cite{bargeunstable2014}. Assume, on the other hand, that $N$ is nonorientable, i.e., a M\"obius strip. Since $K$ is non-saddle and $N$ has only one boundary component it must be either an attractor or a repeller. Consider another copy $N^*$ of $N$ and the flow $\varphi^*=\varphi(\cdot,-t)$ on $M$. We obtain a flow without fixed points on the Klein Bottle by identifying the boundaries of $N$ and $N^*$ via the identity map and considering the flow $\widehat{\varphi}$ which agrees with $\varphi$ in $N$ and with $\varphi^*$ in $N^*$. It is clear that, by its very construction, $\widehat{\varphi}$ extends $\varphi|N$. Now, choosing a point in $\partial N$, either its $\omega$- or its $\omega^*$-limit is a limit cycle contained in $K$  \cite{Demuner}. This limit cycle cannot bound a disk in $N$ since $K$ does not contain fixed points and, as a consequence, it either does not bound any region in $N$ and, hence, it agrees with $K$ or it bounds a M\"obius strip contained in $\mathring{N}$. In this case $K$ must agree with this M\"obius strip and the result follows.
\end{proof}

 The following results are consequences of Theorem~\ref{fixed}.

\begin{corollary}\label{minimal}
Suppose $\varphi$ is a flow defined on a surface and $K$ an isolated invariant continuum which is minimal. Then, $K$ is either a fixed point or a limit cycle.
\end{corollary}

\begin{remark}
If $M$ is compact and the flow is $C^2$, Corollary~\ref{minimal} holds even if we drop the assumption about the isolation as it has been seen in \cite{Schwartz}. 
\end{remark}

\begin{corollary}
If $\varphi$ is a flow defined on a compact surface and every minimal set of $\varphi$ is isolated, then $\varphi$ is topologically equivalent to a $C^\infty$ flow.
\end{corollary}
\begin{proof}
It readily follows from Corollary~\ref{minimal} and Guti\'errez' Theorem.
\end{proof}

\begin{center}
\textbf{Acknowledgements}
\end{center}
\vspace{.15cm}

The author would like to express his gratitude to L. Hern\'andez-Corbato, Jaime J. S\'{a}%
nchez-Gabites and Jos\'e M.R. Sanjurjo for useful comments, inspiring conversations and support during the elaboration of this manuscript. I also would like to thank the referees for their useful suggestions which have helped to improve this manuscript

\bibliographystyle{amsplain}
\bibliography{libreria1}

\end{document}